\documentclass[a4paper,11pt]{article}

\usepackage[a4paper]{geometry} 
\usepackage[utf8]{inputenc}
\usepackage[T1]{fontenc}
\usepackage[english]{babel}
\usepackage{lmodern}
\usepackage{amssymb,amsmath,amsthm,amscd}
\usepackage{mathrsfs}
\usepackage[colorlinks]{hyperref}
\usepackage{tikz}
\usepackage{appendix}
\usepackage{color}

\theoremstyle{plain}
\newtheorem{theorem}{Theorem}[section]
\newtheorem{corollary}[theorem]{Corollary}
\newtheorem{fact}[theorem]{Fact}
\newtheorem{lemma}[theorem]{Lemma}

\theoremstyle{definition}

\theoremstyle{remark}
\newtheorem{remark}[theorem]{Remark}

\numberwithin{equation}{section}

\newcommand{\N}{\mathbb{N}}

\newcommand{\R}{\mathbb{R}}

\newcommand{\ind}[1]{{\mathbf{1}_{\left\{#1\right\}}}}

\newcommand{\floor}[1]{{\left\lfloor #1 \right\rfloor}}
\newcommand{\ceil}[1]{{\left\lceil #1 \right\rceil}}

\newcommand{\calC}{\mathcal{C}}

\DeclareMathOperator{\E}{\mathbf{E}}
\renewcommand{\P}{\mathbf{P}}
\DeclareMathOperator{\Var}{\mathbf{V}\mathrm{ar}}

\newcommand{\calF}{\mathcal{F}}

\newcommand{\calL}{\mathcal{L}}

\renewcommand{\bar}[1]{\overline{#1}}

\newcommand{\egaldistr}{{\overset{(d)}{=}}}

\renewcommand{\rho}{\varrho}
\renewcommand{\epsilon}{\varepsilon}
\renewcommand{\phi}{\varphi}

\newcommand{\T}{\mathbf{T}}

\renewcommand{\hat}[1]{\widehat{#1}}

\newcommand{\rmP}{\mathbb{P}}
\newcommand{\rmE}{\mathbb{E}}
\newcommand{\rmVar}{\mathbb{V}\mathrm{ar}}
\newcommand{\calN}{\mathcal{N}}

\newcommand{\envAS}{\ensuremath{\P\text{-a.s. }}}

\title{Maximal displacement of a supercritical branching random walk in a time-inhomogeneous random environment}
\author{Bastien Mallein\thanks{LAGA, Université Paris 13} ~and Piotr Mi\l{}oś\thanks{Faculty of Mathematics, Informatics and Mechanics, University of Warsaw}}
\date{\today}

\begin{document}

\maketitle
\textcolor{blue}{}\global\long\def\TDD#1{{\color{red}To\, Do(#1)}}
\begin{abstract}
The behavior of the maximal displacement of a supercritical branching random walk has been a subject of intense studies for a long time. But only recently the case of time-inhomogeneous branching has gained focus. The contribution of this paper is to analyze a time-inhomogeneous model with two levels of randomness. In the first step a sequence of branching laws is sampled independently according to a distribution on the set of point measures' laws. Conditionally on the realization of this sequence (called environment) we define a branching random walk and find the asymptotic behavior of its maximal particle. It is of the form $V_n -\varphi \log n + o_\P(\log n)$, where $V_n$ is a function of the environment that behaves as a random walk and $\varphi>0$ is a deterministic constant, which turns out to be bigger than the usual logarithmic correction of the homogeneous branching random walk.
\end{abstract}

\section{Introduction}

We introduce a model of time-inhomogeneous branching random walk on $\R$. Given a sequence $\calL=(\calL_n, n \in \N)$ of point processes laws\footnote{i.e., probability distributions on $\bigcup_{k \in \bar{\N}} \R^k$.}, the time-inhomogeneous branching random walk in the environment $\calL$ is a process constructed as follows. It starts with one individual located at the origin at time 0. This individual dies at time $1$ giving birth to children, that are positioned according to a realization of a point process of law $\calL_1$. Similarly, at each time $n$ every individual alive at generation $n-1$ dies giving birth to children. The position of the children with respect to their parent are given by an independent realization of a point process with law $\calL_{n}$. We denote by $\T$ the (random) genealogical tree of the process. For a given individual $u \in \T$ we write $V(u) \in \R$ for the position of $u$ and $|u|$ for the generation at which $u$ is alive. The pair $(\T,V)$ is called the branching random walk in the time-inhomogeneous environment $\calL$.

The tree $\T$ can be encoded using the well-known Ulam-Harris-Neveu notation, and $V$ is then a random map $\T \to \R$. A precise construction of a time-inhomogeneous branching random walk is presented in \cite[Section 1.1]{Mal15a}.

We assume the time-inhomogeneous Galton-Watson tree $\T$ to be supercritical (the number of individuals alive at generation $n$ grows exponentially fast). In order to simplify the presentation we assume that the number of children of each individual is always at least 1. We take interest in the maximal displacement at time $n$ of $(\T,V)$, defined by 
\[  M_n = \max_{u \in \T : |u|=n} V(u).\]

When the reproduction law does not depend on the time, the asymptotic behavior of $M_n$ has been thoroughly investigated. Hammersley \cite{Ham74}, Kingman \cite{Kin75} and Biggins \cite{Big76} proved that $M_n$ grows at a ballistic speed $v$. Addario-Berry and Reed \cite{ABR09} and Hu and Shi \cite{HuS09} proved the second order of this asymptotic is a negative logarithmic correction in probability. Aïdékon \cite{Aid13} obtained the convergence in law for $M_n-vn+c\log n$, for the correctly chosen constant $c$.

Some recent articles study the effects of a time-inhomogeneous environment on the asymptotic behavior of $M_n$. Fang and Zeitouni \cite{FaZ12a} introduced a branching random walk of length $n$ in which individuals have two children moving independently from their parent's position, with law $\calN(0,\sigma_1^2)$ before time $n/2$ and law $\calN(0,\sigma^2_2)$ between times $n/2$ and~$n$. They proved that for this model, the maximal displacement is again given by a first ballistic order, logarithmic correction plus stochastically bounded fluctuations. More generally, for a branching random walk in which the reproduction law scales with $n$, the correction is expected to be of the larger order $n^{1/3}$ \cite{Mal15b}. Maillard and Zeitouni~\cite{MaZ15} proved that for a branching Brownian motion in time-inhomogeneous environment, the asymptotic behavior is given by a ballistic speed, first correction of order $n^{1/3}$, second correction of order $\log n$ plus stochastically bounded fluctuations.

The main contribution of this article is to study the case of the environment sampled randomly. More precisely, we set $\calL = (\calL_n, n \in \N)$ to be an i.i.d sequence of random variables with the values in the space of laws of point processes. A branching random walk in random environment (BRWre) is a branching random walk with the time-inhomogeneous environment $\calL$. Conditionally on this sequence, we write $\rmP_\calL$ for the law of this BRWre $(\T,V)$ and $\rmE_\calL$ for the corresponding expectation. The joint probability of the environment and the branching random walk is written $\P$, with the corresponding expectation $\E$.

This model of BRWre has been introduced by Biggins and Kyprianou in \cite{BiK04}. Huang and Liu \cite{HuL14} proved that the maximal displacement in the process grows at ballistic speed almost surely, and obtained central limit theorems and large deviations principles for the counting measure of the process. Additional results on the behavior of the so-called bulk of the BRWre may be found in \cite{GLW14,HLL14}, where the growth rate in the bulk of the process is studied, and a central limit theorem and a large deviation principle are derived. Prior to these results, other kind of random environments have been studied for the branching random walk. For example, Baillon, Clément, Greven and den Hollander \cite{BCGH} considered a branching random walk in which the reproduction of individuals depending on their position, instead of the time at which they were alive. Hu and Yoshida \cite{HuY}, as well as other authors, took interest in branching random walks in space-time random environment.

\paragraph*{Notation and assumptions.} Given a sequence $(x_n) \in \R^\N$, we write $O_\P(x_n)$ for a sequence of random variables $(X_n, n \in \N)$ such that $(X_n/x_n)$ is tight. Similarly, $o_\P(x_n)$ denotes a sequence of random variables $(X_n, n \in \N)$ such that $\frac{X_n}{x_n} \to 0$ in $\P$-probability. Moreover $C$ and $c$ stand for two positive constants respectively large enough and small enough, that may change from line to line.

To ensure the non-extinction and non-triviality of the BRWre $(\T,V)$, we assume that
\begin{equation}
  \label{eqn:nonExtinction}
  \rmP_\calL\left( \{u \in \T : |u|=1\} = \emptyset \right) = 0, \quad \envAS
\end{equation}
\begin{equation}
	\text{ and } \quad \P\left(\rmP_\calL\left( \#\{u \in \T : |u|=1\}>1 \right)>0\right)>0.
  \label{eqn:supercritical}
 \end{equation}
By \cite{Tanny:1977aa} one checks that \eqref{eqn:nonExtinction} and \eqref{eqn:supercritical} are sufficient conditions for the random tree $\T$ to be supercritical, and its growth rate is at least exponential. The forthcoming assumption \eqref{eqn:secondMoment} also implies the growth rate of the population is exponential.

For $n \in \N$, we introduce the log-Laplace transform of the point process law $\calL_n$ denoted by $\kappa_n:(0, \infty)\to (-\infty, \infty]$ and given by
\begin{equation}
  \label{eqn:logLaplace}
  \kappa_n (\theta) = \log \rmE_\calL\left( \sum_{\ell \in L_n} e^{\theta \ell} \right),
\end{equation}
where $L_n$ is a point process on $\R$ distributed according to the law $\calL_n$. As the point process $L_n$ is a.s. non-empty we have $\kappa_n(\theta)>-\infty$ \envAS For a fixed $\theta > 0$, $(\kappa_n(\theta), n \in \N)$ is an i.i.d. sequence of random variables under $\P$. We assume that $\E\left(\kappa_1(\theta)_-\right)< \infty$ for all $\theta > 0$, thus we may define $\kappa : \R_+ \to \R \cup \{ \infty\}$ by $\kappa(\theta) = \E\left( \kappa_n(\theta)\right)$.

As $\kappa_n$ is the log-Laplace transform of a measure on $\R$, $\kappa$ is a convex function, $\calC^\infty$ on the interior of the interval $\{\theta > 0 : \kappa(\theta)< \infty\}$. Assuming that this interval is non-empty, we define
\begin{equation}
  \label{eqn:defSpeed}
  v = \inf_{\theta > 0} \frac{\kappa(\theta)}{\theta},
\end{equation}
and we assume there exists $\theta^*>0$ such that $\kappa$ is differentiable at point $\theta^*$ (thus $\kappa$ is finite in a neighborhood of $\theta^ *$) that satisfies
\begin{equation}
  \label{eqn:defTheta}
  \theta^* \kappa'(\theta^*) - \kappa(\theta^*) = 0.
\end{equation}
Under this assumption, we have $v = \kappa'(\theta^*) = \E\left( \kappa_1'(\theta^*) \right)$. We introduce two variance terms, that we assume to be finite
\begin{equation}
  \label{eqn:defVariances}
  	\sigma_Q^2 = {\theta^*}^2 \E\left( \kappa_1''(\theta^*) \right) \in (0, \infty) \quad \mathrm{and} \quad \sigma_A^2 = \mathbf{V}\mathrm{ar}\left( \theta^* \kappa_1'(\theta^*) - \kappa_1(\theta^*) \right) \in [0, \infty).
\end{equation}
Heuristically, the trajectory yielding to the rightmost position at time $n$ can be seen as a random walk path, in the time-inhomogeneous random environment. Then the quantity~$\sigma_Q^2$ represents the average quenched variance of this random walk, while $\sigma_A^2$ represents the annealed variance of the quenched expectation.

Finally, we add two integrability conditions: 
\begin{equation}
  \label{eqn:secondMoment}
  \E\left( \rmE_\calL\left( \left(\sum_{\ell \in L_1} \left( 1 + e^{\theta^* \ell} \right) \right)^2 \right) \right) <  \infty,
\end{equation}
as well as that there exists $C>0$ and $\mu>0$ such that
\begin{equation}
  \label{eqn:integrabilityQuenched}
  \rmE_\calL\left( \sum_{\ell \in L_1} \left(e^{(\theta^*+\mu)\ell} + e^{(\theta^* - \mu) \ell} \right) \right) \leq C e^{\kappa_1(\theta^*)}, \quad \envAS
\end{equation}

The main result of this article is to extend the scope of the result of Hu and Shi \cite{HuS09} and Addario-Berry and Reed \cite{ABR09}, which prove that for a time-homogeneous branching random walk
\begin{equation}
  \label{eqn:homogeneousMain}
  M_n = nv - \frac{3}{2\theta^*} \log n + O_\P(1).
\end{equation}
For $n \in \N$, we write $K_n = \sum_{k = 1}^n \kappa_k(\theta^*)$. This quantity is a random walk depending only on the environment, such that $K_n = \theta^* nv + O_\P(n^{1/2})$. We introduce a function $\gamma:\R_+\to \R_+$  given by the following limit
\begin{equation}
  \label{eq:defgamma}
  \lim_{t \to  \infty} \frac{1}{\log t} \log \P\left( \left. B_s + 1 \geq \beta W_s, s \leq t \right| W \right) = - \gamma(\beta),\quad \envAS
\end{equation}
where $B,W$ are independent standard Brownian motions. This function and similar problems are studied in detail in the companion paper \cite{Mallein:2015aa}. In particular, according to \cite[Theorem 1.1]{Mallein:2015aa}, the function $\gamma$ is well-defined, convex and strictly increasing on $(0, \infty)$.

\begin{theorem}
\label{thm:logCorrection}
Under the assumptions \eqref{eqn:nonExtinction}, \eqref{eqn:supercritical},  \eqref{eqn:defTheta}, \eqref{eqn:defVariances}, \eqref{eqn:secondMoment} and \eqref{eqn:integrabilityQuenched}, writing 
\[
	\phi = \frac{2}{\theta^*} \gamma \left(\tfrac{\sigma_A}{\sigma_Q}\right) + \frac{1}{2\theta^*},
\]
we have
\[
  \lim_{n \to  \infty} \rmP_\calL\left( M_n - \tfrac{1}{\theta^*} K_n \geq -\beta \log n\right) =
  \begin{cases}
    1 & \mathrm{if} \quad \beta > \phi\\
    0 & \mathrm{if} \quad \beta < \phi\\
  \end{cases}
  \qquad \text{in } \P\text{-probability}.
\]
\end{theorem}

When the reproduction law does not depend on the time we have $\sigma^2_A = 0$. As $\gamma(0)=1/2$,  this theorem is consistent with the results of Hu and Shi, and Addario-Berry and Reed. Moreover, as $\gamma$ is increasing, as soon as $\sigma_A^2>0$ the logarithmic correction is larger than the time-homogeneous one.

\begin{remark}
In a homogeneous branching random walk with (fixed) reproduction law $\calL_1$ by standard results we have $\lim_{n \to  \infty} \frac{M_n}{n} = \inf_{\theta > 0} \frac{\kappa_1(\theta)}{\theta} =: v_1$. By Theorem~ \ref{thm:logCorrection}, $\lim_{n \to  \infty} \frac{M_n}{n} = v$ in probability for a BRWre. We observe that
\[
  v = \inf_{\theta > 0} \frac{\E\left( \kappa_1(\theta)\right)}{\theta} \geq \E\left(\inf_{\theta > 0} \frac{\kappa_1(\theta)}{\theta} \right) = \E(v_1).
\]
This inequality is strict as soon as $\sigma_A^2>0$. Therefore, for a typical non-degenerate random environment, the speed of the branching random walk in random environment is larger than the expected speed of a branching random walk with law~$\calL_1$. This behavior is similar to the one observed in the branching random walk with increasing variance in \cite{FaZ12a}, the optimal path can take advantage of the inhomogeneities of the environment to reach further points.
\end{remark}

A direct consequence of this theorem is the asymptotic behavior of $M_n$ under law $\P$.
\begin{corollary}
Under the assumptions of Theorem \ref{thm:logCorrection}, we have
\[
  \lim_{n \to  \infty} \frac{M_n - \tfrac{1}{\theta^*} K_n}{\log n} = -\phi \quad\text{in } \P \otimes \rmP_\calL \text{-probability}.
\]
\end{corollary}

Most likely, the convergence cannot be strengthened to an almost sure result of the form $M_n = \frac{1}{\theta^*} K_n - \phi \log n + o_{\rmP_\calL}(\log n)$ \envAS In fact, introducing the (quenched) median of $M_n$:
\begin{equation}
  \label{eqn:defQuenchedMedian}
  m^Q_n = \sup\left\{ a \in \R : \rmP_\calL(M_n \geq a) \geq 1/2 \right\},
\end{equation}
we believe that $m^Q_n$ exhibits non-trivial $\log n$-scale fluctuations in $\P$-probability. Moreover, using the results of Fang \cite{Fan12}, it is established that under some additional assumptions $(M_n-m^Q_n)$ is tight. We discuss this tightness issue for BRWre in Section \ref{sec:tightness}.

The proof strategy for Theorem \ref{thm:logCorrection} is very similar to method used to study the maximal displacement in an homogeneous branching random walk, which can for example be found in \cite{AiS10}. We first introduce a result that links the the moments of additive functionals of the branching random walk with random walk estimates: the many-to-one lemma (Lemma~\ref{lem:manytoone}). However, the randomness of the environment implies that the random walk obtained here is a random walk in (time-inhomogeneous) random environment.

Using the many-to-one lemma, we then compute the first and second moment of the number of particles making a suitable excursion below a well-chosen barrier. This barrier is chosen such that the first and second moment are of the same order, therefore Markov and Paley-Zygmund inequalities allow to bound from above and from below the probability for $M_n$ to be larger than $\tfrac{1}{\theta^*}K_n - \phi \log n$. Finally, using a standard truncation argument we are able to bound with high probability the position of $M_n$ at time $n$.

The main impact of the random environment is that the random walk obtained in the many-to-one lemma is a random walk in random environment. Hence, one has to study the probability for a random walk to realize an excursion of length $n$ above a suitable barrier, which is done in Theorem~\ref{thm:randomExcursion}. This probability behaves as $n^{-\lambda + o_\P(1)}$, with $\lambda$ a constant that depends only of $\frac{\sigma_A}{\sigma_Q}$. This constant is directly related to the logarithmic correction of the BRWre, through the equation $\phi = \frac{\lambda}{\theta^*}$. This is similar to what happens in the time-homogeneous case: the probability for a random walk to make an excursion of length $n$ is of order $n^{-3/2}$, and the logarithmic correction of the branching random walk is $\frac{3}{2\theta^*}$.

The rest of the article is organized as follows. We introduce in the next Section \ref{sec:tightness} the many-to-one lemma, and in particular the time-inhomogeneous random walk associated to the BRWre. In Section \ref{sec:excursion}, we compute the probability for this random walk in random environment to make an excursion of length $n$. We use this result to prove Theorem \ref{thm:logCorrection} in Section \ref{sec:corrections}, using the method described above.

\section{The many-to-one lemma}
\label{sec:many-to-one}

We introduce the celebrated many-to-one lemma. It has been essential in studies of extremal behavior of branching random walks. It can be traced down to the early works of Peyrière \cite{Pey74} and Kahane and Peyrière \cite{KaP76}. Many variations of this result have been introduced, see e.g. \cite{BiK04}. In this article, we use a time-inhomogeneous version of this lemma, that can be found in \cite[Lemma 2.2]{Mal15a}. For all $n \geq 1$, we write $L_n$ for a realization of the point process with law $\calL_n$, and we define the probability measure $\mu_n$ on $\R$ by
\begin{equation}
  \label{eqn:defMun}
  \forall x \in \R, \quad \mu_n((-\infty,x]) = \rmE_\calL\left( \sum_{\ell \in L_n} \ind{\ell \leq x}e^{\theta^* \ell - \kappa_n(\theta^*)} \right),
\end{equation}
where we recall that $\kappa_n$ is the log-Laplace transform of $L_n$ given by \eqref{eqn:logLaplace} and $\theta^*$ by \eqref{eqn:defTheta}. Up to a possible enlargement of the probability space, we define a sequence $(X_n, n \in \N)$ of independent random variables, where $X_n$ has law $\mu_n$. We set $S_n = \sum_{j=1}^n X_j$. From now on, $\rmP_\calL$ stands for the joint law of the BRWre $(\T,V)$ and the random walk in random environment $S$, conditionally on the environment $\calL$.
\begin{lemma}[Many-to-one lemma]
\label{lem:manytoone}
For any $n \in \N$ and any measurable non-negative function $f : \R^n \to \R$, we have
\begin{equation}
  \label{eqn:manytoone}
  \rmE_\calL \left( \sum_{|u|=n} f(V(u_1),\ldots, V(u_n)) \right) = \rmE_\calL\left( e^{-\theta^* S_n + \sum_{k=1}^n \kappa_j(\theta^*)} f(S_1,\ldots S_n) \right), \quad \envAS
\end{equation}
\end{lemma}
To simplify notation, we introduce the process $(T_n, n\geq 0)$ given by
\begin{equation}
  \label{eqn:defT}
   T_n = \theta^* S_n - K_n.
\end{equation}
Now \eqref{eqn:manytoone} has a more compact form
\begin{equation}
  \label{eqn:manymodified}
  \rmE_\calL\left( \sum_{|u|=n} f\left(\theta^* V(u_j)-K_j, j \leq n\right) \right) = \rmE_\calL\left( e^{-T_n} f(T_j, j \leq n) \right).
\end{equation}
The process $T$ is a random walk in random environment: conditionally on the environment~$\calL$, $T_n$ is the sum of $n$ independent random variables, and the law of $T_j-T_{j-1}$ only depends on $\calL_j$. Note that under law $\P$, $T$ is simply a random walk.
\begin{remark}
\label{rem22}
Note that the process we call ``random walk in random environment'' is not a typical random walk in random environment. Indeed, this process is ``trivial'', in the sense that each random environment is used exactly once. In typical random walks in random environment, the random environment is either purely spatial or space-time. However, despite looking trivial, the probability for $T$ to make an excursion of length $n$ above $0$, conditionally on the random environment yields a non-trivial exponent inherited from the randomness of the environment.
\end{remark}

By \eqref{eqn:defTheta}, we have $\E(T_n)=0$. Moreover by \eqref{eqn:defMun}, we have
\[
  \rmE_\calL(X_j) = \kappa'_j(\theta^*) \text{ and } \rmVar_\calL(X_j) = \kappa''_j(\theta^*).
\]
Thus \eqref{eqn:defVariances} can be rewritten
\[
  \sigma_Q^2 = \E\left( \rmVar_\calL(T_1) \right) \quad \mathrm{and} \quad \sigma_A^2 = \Var\left( \rmE_\calL(T_1) \right).
\]
This confirms the heuristic description of $\sigma_Q^2$ as the mean of a quenched variance, ans $\sigma_A^2$ as the variance of a quenched mean.

Using the many-to-one lemma, we can prove that with high probability, every individual in the BRWre stays at all time below the environment-dependent path $n \mapsto \frac{1}{\theta^*} K_n + y$, for $y$ large enough. This explains why a key tool in the study of the asymptotic behavior of $M_n$ is the random ballot theorem, discussed in the next section.
\begin{lemma}
\label{lem:frontier}
Under the assumption \eqref{eqn:defTheta}, for any $y > 0$ we have
\[
  \rmP_\calL\left( \exists u \in \T : V(u) > \tfrac{1}{\theta^*} K_{|u|} + y \right) \leq e^{-\theta^* y},\quad \envAS
\]
\end{lemma}

\begin{proof}
Using \eqref{eqn:manymodified}, the following classical upper bound holds for any branching random walk in time-inhomogeneous environment:
\begin{align*}
  &\rmP_\calL\left( \exists u \in \T : V(u) > \tfrac{1}{\theta^*} K_{|u|} + y \right)\\
  &\qquad \qquad \qquad\leq \rmE_\calL\left( \sum_{u \in \T} \ind{ V(u) - \tfrac{1}{\theta^*} K_{|u|} > y} \ind{ V(u_j) - \tfrac{1}{\theta^*} K_j \leq y, j < |u| } \right)\\
  &\qquad \qquad \qquad\leq \sum_{n=1}^{ \infty} \rmE_\calL\left( \sum_{|u|=n} \ind{\theta^* V(u) - K_n > \theta^* y} \ind{\theta^* V(u_j) - K_j \leq \theta^* y, j < n} \right)\\
  &\qquad \qquad \qquad\leq \sum_{n=1}^{ \infty} \rmE_\calL\left( e^{-T_n} \ind{T_n > \theta^* y} \ind{T_j \leq \theta^* y, j < n} \right)\\
  &\qquad \qquad \qquad\leq e^{-\theta^* y} \rmP_\calL\left( \exists n \in \N : T_n > \theta^* y \right) \leq e^{-\theta^* y}, \quad \envAS \qedhere
\end{align*}
\end{proof}

Due to the time-inhomogeneity, it will be useful to consider time-shifts of the environment, and a shifted version of \eqref{eqn:manytoone}. For $k \in \N$,  write $\rmP^k_\calL$ for the law of the branching random walk with the random environment $(\calL_{j+k},j \in \N)$. By convention we assume that under $\rmP^k_\calL$, the random walk $S$ is $S_n=X_{k+1}+\cdots + X_{k+n}$. In this scenario \eqref{eqn:manytoone} writes as
\begin{equation}
  \label{eqn:manytooneGeneral}
  \rmE^k_\calL \left( \sum_{|u|=n} f(V(u_1),\ldots, V(u_n)) \right) = \rmE^k_\calL\left( e^{-\theta^* S_n + \sum_{j=k+1}^{k+n} \kappa_j(\theta^*)} f(S_1,\ldots S_n) \right), \quad \envAS
\end{equation}
Similarly, we assume that under $\rmP^k_\calL$ the process $T$ is given by $T_n = \theta^* S_n - \sum_{j=k+1}^{k+n} \kappa_j(\theta^*)$.

\section{Ballot theorem for a random walk in random environment}
\label{sec:excursion}

It is well-known, see for example \cite{AiS10}, that the constant $\frac{3}{2}$ in \eqref{eqn:homogeneousMain} is directly related to the exponent of the ballot theorem problem i.e., for any centered random walk $(S_n)$ with finite variance
\[
  \P(S_n \geq -\log n, S_j \leq 0 , j \leq n) \approx \frac{C}{n^{3/2}}.
\]
For review of ballot-like theorems, one can look at \cite{ABR08}. To prove Theorem \ref{thm:logCorrection}, we compute similar quantities for random walks in random environment. More precisely, we compute the probability for $T$ to make a negative excursion of length $n$, conditionally on the environment $\calL$. We devoted a companion paper \cite{Mallein:2015aa} to study the asymptotic behavior of the probability for a random walk in random environment to stay positive during $n$ units of time and to study the function $\gamma$, defined in \eqref{eq:defgamma}. We start with a translation of \cite[Theorem 1.11]{Mallein:2015aa} to our notation.

We recall again that in this article, when we consider random walks in random environment that are not the typical case of ``space-not-time'' random environments, in which a given reproduction law is used multiple times (cf Remark~\ref{rem22}). For typical random walks in random environment, the process considered is a time-homogeneous Markov process whose semigroup is random, while the random walks in random environment we study are sums of independent random variables.

\begin{theorem}\label{thm:BRWBasicExtended}
Let $T$ be as in \eqref{eqn:defT}, $(x_n) \in \R_+^\N$ and $(f_n) \in \R^\N$ such that
\[
  \lim_{n \to  \infty} x_n =  \infty,\ \lim_{n \to  \infty} \frac{\log x_n}{\log n}=0 \text{ and } \limsup_{n\to  \infty} \frac{|f_n|}{n^{1/2-\epsilon}}=0,
\]
for some $\epsilon > 0$. For any $0 \leq a < b \leq  \infty$ we have
\begin{equation*}
  \lim_{n\to \infty} \frac{\log\rmP_\calL\left( x_n + T_n \in [a n^{1/2},bn^{1/2}], x_n + T_j \geq f_j, j \leq n \right)}{\log n} = -\gamma\left(\tfrac{\sigma_A}{\sigma_Q}\right),\quad \envAS
\end{equation*}
\end{theorem}

The proof of this theorem is rather long, but we give here a quick description of the main steps of proof. First, using KMT-type coupling between random walks and Brownian motions, we prove that
\[
  \rmP_\calL\left( T_j \geq 0, j \leq n\right) \approx \P\left( \sigma_Q B_s \geq \sigma_A W_s - 1, s \leq n \middle| W \right),
\]
where $B$ and $W$ are two independent Brownian motions, the Brownian motion $\sigma_A W$ approaching the random walk (with respect to the environment) $(\rmE_\calL(T_n), n \in \N)$ and $\sigma_Q B$ replacing the sum of independent centered random variables $(T_n - \rmE_\calL(T_n), n \in \N)$. The convergence in \eqref{eq:defgamma} then proves Theorem \ref{thm:BRWBasicExtended}.

This convergence is a consequence of the fact that
\[
  \left( \log  \P\left(B_s \leq \beta W_s, 1 \leq s \leq e^t \middle| W \right), t \geq 0 \right)
\]
is a subadditive sequence, hence Kingman's subadditive ergodic theorem guarantees the existence and finiteness of $\gamma(\beta)$. Moreover, one can observe that
\[
  \P\left(B_s \leq \beta W_s, e^u \leq s \leq e^{t+u} \middle| W \right) \egaldistr \P\left(B_s \leq \beta W_s, 1 \leq s \leq e^t \middle| W \right),
\]
therefore the behavior of $ \log  \P\left(B_s \leq \beta W_s, 1 \leq s \leq e^t \middle| W \right)$ should be similar to the sum of $t$ i.i.d. random variables. Hence, we conjecture the existence a central limit type theorem can be expected, of the form
\begin{equation}
  \label{eqn:conjecture}
  \lim_{t \to \infty} \frac{\log \P\left(B_s \leq \beta W_s, 1 \leq s \leq e^t \middle| W \right) - \gamma(\beta) \log t }{(\log t)^{1/2}} = \mathcal{N}(0, \rho(\beta)) \quad \text{in law.}
\end{equation}

Observe that the statement of Theorem \ref{thm:BRWBasicExtended} also holds for $-T$. More precisely, under the same assumptions, \envAS we have
\begin{equation}
\lim_{n\to \infty}\frac{\log\rmP_\calL\left(-x_n+T_n\in[-bn^{1/2}, -a n^{1/2},], -x_n+T_{j}\leq f_j, j\leq n \right)}{\log n} = -\gamma\left(\tfrac{\sigma_A}{\sigma_Q}\right),\label{eq:extendedRWoverRWclaimReversed}
\end{equation}

\begin{remark}
The results obtained in Theorem \ref{thm:BRWBasicExtended} and \eqref{eq:extendedRWoverRWclaimReversed} can be rephrased as
\[
  \rmP_\calL\left( T_j \geq 0, j \leq n \right) = n^{- \gamma(\tfrac{\sigma_A}{\sigma_Q})(1 + o(1))}, \text{ as } n \to \infty, \enskip \envAS
\]
If the result \eqref{eqn:conjecture} holds, then this result cannot be strengthened to a ballot-type theorem similar to the homogeneous case, i.e. the existence of a (random) constant $C>0$ such that the probability for $T$ to stay positive is bounded by $C n^{-\gamma}$ for $n$ large enough.\end{remark}

The exponent $\gamma:=\gamma\left(\tfrac{\sigma_A}{\sigma_Q}\right)$ is used to compute the probability for $T$ to make a negative excursion of length $n$. To do so, we observe that an excursion of length $n$ can be divided into three parts. Between times $0$ and $n/3$, the random walk stays negative, which happens with probability $n^{-\gamma}$. Similarly, the end part between times $2n/3$ and $n$, seen from backward is a random walk in random environment staying positive, happening with probability $n^{-\gamma}$. Finally, the part between times $n/3$ and $2n/3$ is a random walk joining these two paths, which has probability $n^{-1/2}$ by local CLT. Consequently, writing
\begin{equation}
  \label{eqn:defLambda}
  \lambda := 2 \gamma\left(\tfrac{\sigma_A}{\sigma_Q}\right) + \frac{1}{2},
\end{equation}
we expect $\rmP_\calL\left( T_n \leq 1, T_j \geq 0, j \leq n\right) \approx n^{-\lambda}$. However, working with random environments yields additional difficulties.
\begin{theorem}
\label{thm:randomExcursion}
Let $T$ be as in \eqref{eqn:defT}. We set $(x_n), (a_n) \in \R_+^\N$ such that
\[
  a_n \leq x_n, \ \liminf_{n \to  \infty} \frac{a_n}{\log n} > 0 \text{  and  } \lim_{n \to  \infty} \frac{\log x_n}{\log n} = 0.
\]
Let $\alpha \in [0,1/2)$, we write $r_{n,j} = \min(j,n-j)^\alpha$ for $0 \leq j \leq n$. For any $\epsilon > 0$ we have
\begin{equation}
  \label{eqn:randomUpperExcursion}
  \lim_{n \to  \infty} \P\left( \sup_{x , y \in [0, x_n]} \sup_{k,k' \leq x_n} \rmP^k_\calL\left( T_{n-k-k'} \geq x-y, T_j \leq x, j \leq n-k-k' \right) \leq n^{- \lambda+\epsilon} \right) = 1,
\end{equation}
\begin{equation}
  \label{eqn:randomLowerExcursion}
  \lim_{n \to  \infty} \P\left( \inf_{x,y \in [a_n, x_n]} \rmP_\calL\left( T_{n} \geq x-y, T_j \leq x-r_{n,j}, j \leq n \right) \geq n^{- \lambda - \epsilon} \right) = 1.
\end{equation}
\end{theorem}

We stress that the mode of convergence in Theorem \ref{thm:randomExcursion} cannot be strengthen to a \envAS result. We will see that the convergence in the forthcoming \eqref{eq:backward} holds only in probability. For the moment, as an illustration, we present a simple lemma concerning an analogous phenomenon for the Brownian motion over a Brownian motion path.
	
\begin{lemma}
\label{lem:backward}
For all $\beta \in \R$, we have
\begin{align}
  &\lim_{t \to  \infty} \frac{1}{\log t}\log \P\left(\left. B_s+1 \geq \beta( W_t - W_{t-s}), s \leq t \right|W \right) = -\gamma(\beta) \quad \text{in } \P {-probability}\nonumber\\
  &\limsup_{t \to  \infty} \frac{1}{\log t}\log \P\left(\left. B_s+1 \geq \beta(W_t - W_{t-s}), s \leq t \right|W \right)\geq -1/2, \quad \envAS \label{eqn:limsup}\\
  &\liminf_{t \to  \infty} \frac{1}{\log t}\log \P\left(\left. B_s+1 \geq \beta(W_t - W_{t-s}), s \leq t \right|W \right)\leq -\beta^2/2, \quad \envAS \label{eqn:liminf}
\end{align}
\end{lemma}

Note the constants we obtain in \eqref{eqn:limsup} and \eqref{eqn:liminf} are non-optimal. In particular, \eqref{eqn:liminf} gives no information for $\beta \in (-1,1)$. The reason of fluctuations on $\log n$ scale is that the environment ``seen from backward'' every awhile is particularly favorable or  unfavorable for the random walk to stay non-negative. More precisely, one can make the following observation: almost surely, infinitely often the Brownian motion seen from backward stays negative, or grows faster than $t$ during $\log t$ units of time. As this result plays no role in the rest of the article, we skip the detailed proof.

To prove this theorem, we first observe that Theorem \ref{thm:BRWBasicExtended} can be easily extended to obtain uniform upper and lower bounds.
\begin{lemma}
\label{lem:uniformRwOverRw} Let $(x_n),(t_n) \in \R_+^\N$ be such that
\[  \lim_{n \to  \infty} x_n = \lim_{n \to  \infty} t_n =  \infty, \ \lim_{n \to  \infty} \frac{\log x_n}{\log n} = \lim_{n \to  \infty} \frac{\log t_n}{\log n} = 0
\]
For any $\alpha \in [0,1/2)$, we have
\[
  \lim_{n \to  \infty} \sup_{k \leq t_n} \frac{\log \rmP^k_{\calL}\left(T_j \geq -x_n - j^\alpha, j \leq n\right)}{\log n} = - \gamma\left(\frac{\sigma_A}{\sigma_Q}\right), \quad \envAS
\]
\end{lemma}

\begin{proof}
The lower bound of this lemma is a direct consequence of Theorem~\ref{thm:BRWBasicExtended}, we only consider the upper bound in the rest of the proof. For any $k \leq t_n$, applying the Markov property at time $k$ we obtain
\begin{multline*}
  \rmP_\calL\left( T_j \geq - x_n - \log t_n - j^\alpha, j \leq n \right) \\
  \geq \rmP_\calL\left( T_j \geq -\log t_n, j \leq k\right) \times \rmP^k_\calL\left( T_j \geq - x_n - j^\alpha, j \leq n-k \right).
\end{multline*}
As a consequence for any  $k \leq t_n$ we have
\begin{align}
  \rmP^k_\calL \left( T_j \geq - x_n - j^\alpha, j \leq n-k \right) &\leq \rmP^k_\calL \left( T_j \geq - x_n - j^\alpha, j \leq n -t_n \right)\nonumber\\
  &\leq \frac{\rmP_\calL\left( T_j \geq - x_n - \log t_n -j^\alpha, j \leq n \right)}{\rmP_\calL\left( T_j \geq -\log t_n, j \leq t_n \right)} \label{eqn:estimate}. 
\end{align}
By Theorem \ref{thm:BRWBasicExtended} we get
\[
  \lim_{n \to  \infty} \frac{\log \rmP_\calL\left( T_j \geq - x_n -\log t_n - j^\alpha, j \leq n \right)}{\log n} = - \gamma\left(\frac{\sigma_A}{\sigma_Q}\right), \quad \envAS
\]
Similarly, as $\lim_{n} \frac{\log t_n}{\log n} = 0$ we have
\[
  \lim_{n \to  \infty} \frac{\log \rmP_\calL \left( T_j \geq -\log t_n, j \leq t_n \right)}{\log n} = 0, \quad  \envAS
\]
Applying these two estimates to \eqref{eqn:estimate} concludes the proof.
\end{proof}

We now derive from the Berry-Esseen theorem~\cite{Ber41,Ess42} a local limit theorem for random walks in random environment. There exists a universal constant $C_1$ such that, given $X_1,\ldots X_n$ independent centered random variables with finite third moment and $N$ is a standard Gaussian random variable
\begin{equation}
  \label{eqn:BerryEssen}
  \sup_{t \in \R} \left| \P\left( \frac{ \sum_{i=1}^nX_i}{\sqrt{\sum_{i=1}^n\Var(X_i)}} \geq t \right) - \P(N \geq t) \right| \leq C_1 \frac{\max_{i \leq n} \frac{\E(|X_i|^3)}{\Var(X_i)}}{\sqrt{\sum_{i=1}^n \Var(X_i)}}.
\end{equation}
We obtain the following bounds.
\begin{lemma}
Let $(x_n) \in \R_+^\N$ such that 
\[ \lim_{n \to  \infty} x_n =  \infty \text{ and } \frac{\log x_n}{\log n} = 0.\]
For any $\epsilon > 0$, we have
\begin{equation}
  \label{eqn:lltBoundFromAbove}
\lim_{n\to \infty}\P\left(\sup_{y\in\R}\frac{\log\rmP_{\calL}(T_{n}\in[y,y+x_n])}{\log n}\leq-\frac{1}{2}+\epsilon\right)=1.
\end{equation}
Moreover, for any $\epsilon > 0$ and $a > 0$,
\begin{equation}
  \label{eqn:lltBoundFromBelow}
  \lim_{n\to \infty}\P\left(\inf_{|y|< an^{1/2}}\frac{\log\rmP_{\calL}(T_{n}\in[y,y+x_n])}{\log n}\geq-\frac{1}{2}-\epsilon\right)=1.
\end{equation}
\end{lemma}

\begin{proof}
For any $n \in \N$, we set $\Sigma_n = \rmVar_\calL(T_n)$. By \eqref{eqn:BerryEssen} and \eqref{eqn:integrabilityQuenched}, we have
\begin{equation}
  \label{eqn:berwre}
  \sup_{t \in \R} \left| \rmP_\calL\left( T_n \geq t  \right) - \rmP_\calL\left(N \geq \frac{t-\rmE_\calL(T_n)}{\sqrt{\Sigma_n}}\right)\right| \leq \frac{C_2}{\sqrt{\Sigma_n}},\quad \envAS
\end{equation}
where $N$ is a Gaussian random variable independent of $\calL$ and $C_2>0$. 

Using \eqref{eqn:berwre}, for any $y \in \R$ we have
\begin{align*}
  \rmP_\calL\left( T_n \in [y,y+x_n] \right) &= \rmP_\calL\left( \frac{T_n}{\sqrt{\Sigma_n}} \geq \frac{y}{\sqrt{\Sigma_n}}\right) - \rmP_\calL\left( \frac{T_n}{\sqrt{\Sigma_n}} > \frac{y+x_n}{\sqrt{\Sigma_n}} \right)\\
  &\leq \rmP_\calL\left( N + \frac{\rmE_\calL(T_n)}{\sqrt{\Sigma_n}} \in \left[\tfrac{y}{\sqrt{\Sigma_n}},  \tfrac{y +x_n}{\sqrt{\Sigma_n}}\right] \right) + \frac{2C_2}{\sqrt{\Sigma_n}}\\
  &\leq C \frac{x_n}{\sqrt{\Sigma_n}}.
\end{align*}
Observing that by the law of large numbers we have $\lim_{n \to  \infty} \frac{\Sigma_n}{n} = \sigma^2_Q$ \envAS we conclude that
\[
  \limsup_{n \to  \infty} \sup_{y \in \R} \frac{\log \rmP_\calL\left( T_n \in [y,y+x_n] \right)}{\log n} \leq -\frac{1}{2},\quad \envAS
\]
which yields \eqref{eqn:lltBoundFromAbove}.

Similarly, for any $y \in \R$ with $|y| \leq an^{1/2}$, \eqref{eqn:berwre} yields
\begin{align*}
  \rmP_\calL\left( T_n \in [y,y+x_n]\right)
  &\geq \rmP_\calL\left( N + \frac{\rmE_\calL(T_n)}{\sqrt{\Sigma_n}} \in \left[\tfrac{y}{\sqrt{\Sigma_n}},  \tfrac{y + x_n}{\sqrt{\Sigma_n}}\right] \right)\\
  &\geq c \frac{x_n}{\sqrt{\Sigma_n}} \exp\left( -\left( \frac{ an^{1/2} + \left|\rmE_\calL(T_n)\right|}{\sqrt{\Sigma_n}}  \right)^2 \right) - \frac{2C_2}{\mu \sqrt{\Sigma_n}}.
\end{align*}
Under $\P$ the random variable $\frac{\left|\rmE_\calL(T_n)\right|}{\sqrt{\Sigma_n}}$ converges in law toward $|\calN(0,\tfrac{\sigma_A^2}{\sigma_Q^2})|$, thus for any $\epsilon>0$ we have
\[
  \lim_{t \to  \infty} \P\left( \inf_{|y| < an^{1/2}} \frac{\log \rmP_\calL(T_n \in [y,y+x_n])}{\log n} \geq -\frac{1}{2}-\epsilon \right) = 1.
\]
\end{proof}

Equation \eqref{eqn:lltBoundFromBelow} can be made more precise, to compute the probability for a random walk in random environment to end up in a given interval, while staying below a wall at level $O(n^{1/2})$.
\begin{lemma}
\label{lem:lltBelowWithWall}
Let $(x_n) \in \R_+^\N$ such that $\lim_{n \to  \infty} x_n =  \infty$. For any $a>0$, $\epsilon>0$ and $\eta > 0$, there exists $b>0$ such that
\[
  \liminf_{n \to  \infty} \P\left( \inf_{x \in [-a n^{1/2},a n^{1/2}]}\frac{\log \rmP_\calL(T_n \in [x,x+x_n], T_j \leq b n^{1/2}, j \leq n)}{\log n} > -\frac{1}{2}-\epsilon \right) \geq 1-\eta.
\]
\end{lemma}

\begin{proof}
To compute this probability, we observe that for any $|x| \leq an^{1/2}$,
\begin{multline*}
  \rmP_\calL(T_n \in [x,x+x_n], T_j \leq b n^{1/2}, j \leq n)\\ = \rmP_\calL\left( T_n \in [x,x+x_n] \right) -  \rmP_\calL\left( T_n \in [x,x+x_n], \tau_n \leq n \right),
\end{multline*}
where $\tau_n = \inf\{ k \in \N : T_k \geq b n^{1/2} \}$. Moreover, recall that
\[
  \rmP_\calL\left( T_n \in [x,x+x_n] \right) \geq c \frac{x_n}{\sqrt{\Sigma_n}} \exp\left( -\left( \frac{ an^{1/2} + \left|\rmE_\calL(T_n)\right|}{\sqrt{\Sigma_n}}  \right)^2 \right).
\]
We now bound $\rmP_\calL\left( T_n \in [x,x+x_n], \tau_n \leq n \right)$ from above.

Let $\eta > 0$ we set $b>2a$ such that for any $n \in \N$
\[
  \rmP_\calL\left( \max_{j \leq n} \left| \rmE_\calL(T_j) \right| \leq bn^{1/2}/2 \right) > 1-\eta.
\]
Thus, on the event $\left\{\max_{j \leq n} \left| \rmE_\calL(T_j) \right| \leq bn^{1/2}/2\right\}$, we have
\begin{multline*}
  \rmP_\calL\left( T_n \in [x,x+x_n], \tau_n \leq n \right)
  \leq \max_{k \geq n/2} \sup_{y \geq bn^{1/2}/2} \rmP_\calL\left( T_k-\rmE_\calL(T_k) \in [y,y+x_n]\right)\\
   + \max_{k \leq n/2} \sup_{y \geq bn^{1/2}/2} \rmP^k_\calL\left( T_{n-k}-\rmE^{n-k}_\calL(T_{n-k}) \in [y,y+x_n]\right).
\end{multline*}
By the Berry-Esseen inequality, for $n$ large enough, we obtain
\[
  \rmP_\calL\left( T_n \in [x,x+x_n], \tau_n \leq n \right) \leq C \frac{x_n}{\sqrt{\Sigma_n}} e^{-b^2/4\sigma_Q^2},\quad \envAS
\]
we conclude that
\[
  \liminf_{n \to  \infty} \P\left( \inf_{x \in [-a n^{1/2},a n^{1/2}]}\frac{\log \rmP_\calL(T_n \in [x,x+x_n], T_j \leq b n^{1/2}, j \leq n)}{\log n} > -\frac{1}{2}-\epsilon \right) \geq 1-\eta.
\]
\end{proof}

Using the above results, we prove Theorem \ref{thm:randomExcursion}, starting from the upper bound.
\begin{proof}[Proof of \eqref{eqn:randomUpperExcursion}]
We decompose the path of $T$ into three pieces. We set $p = \floor{n/3}$, applying the Markov property at time $p$, we have
\begin{multline}
  \rmP^k_\calL\left( T_{n-k-k'} \geq x-y, T_j \leq x, j \leq n-k-k' \right)\\ \leq \rmP^k_\calL \left( T_j \leq  x_n, j \leq p-k \right) \sup_{z \in \R} \rmP^p_\calL\left( T_{n-k'-p} + z \geq x-y, T_j+z \leq x, j \leq n-p-k' \right).\label{eqn:decomp}
\end{multline}
We introduce $\hat{T}^{n-p-k'}_j = T_{n-p-k'}-T_{n-p-k'-j}$ the time reversed random walk. Using the fact that $T_{n-p-k'}+z \in [x-y,x]$ and $y\leq x_n$, we have
\begin{multline*}
  \rmP^p_\calL\left(T_{n-p-k'} + z \geq x-y, T_j+z \leq x, j \leq n-p-k'\right) \\
  \leq \rmP^p_\calL \left(\hat{T}_{n-p-k'}^{n-p-k'}\in [x-y-z, x-z],  \hat{T}_j^{n-p-k'}\geq -x_n,j\leq n-p-k'\right).
\end{multline*}
Applying the Markov property to $\hat{T}$ at time $p-k'$ yields
\begin{multline*}
  \rmP^p_\calL \left( T_{n-p-k'} + z \geq x-y, T_j+z \leq x, j \leq n-p-k' \right)\\
  \leq \rmP^p_\calL \left( \hat{T}_j^{n-p-k'} \geq -x_n, j \leq p-k' \right) \sup_{h \in \R} \rmP^p_\calL \left( \hat{T}_{n-p-k'}^{n-p-k'} - \hat{T}_{p-k'}^{n-p-k'} \in [h,h+x_n] \right).
\end{multline*}
A direct calculation shows that $\hat{T}_{n-p-k'}^{n-p-k'} - \hat{T}_{p-k'}^{n-p-k'} = T_{n-2p}$, therefore \eqref{eqn:decomp} yields
\begin{multline}
  \label{eqn:decompositionUpperFinal}
  \rmP^k_\calL \left( T_{n - k - k'} \geq x-y, T_j \leq y, j \leq n -k - k' \right) \leq \rmP^k_\calL \left( T_j \leq x_n, j \leq p-k  \right)\\
  \times \rmP^p_\calL\left( \hat{T}^{n-p-k'}_j \geq -x_n, j \leq p-k' \right)  \sup_{h \in \R} \rmP^p_\calL\left( T_{n-2p} \in [h,h+x_n] \right).
\end{multline}
We conclude that
\begin{multline*}
  \sup_{0 \leq x , y \leq x_n} \sup_{k,k' \leq x_n} \rmP^k_\calL \left( T_{n - k - k'} \leq x, T_j \geq y, j \leq n -k - k' \right)\\
  \leq \sup_{k \leq x_n} \rmP^k \left( T_j \geq - x_n, j \leq p-x_n \right) \sup_{k' \leq x_n} \rmP^p\left( \hat{T}^{n-p-k'}_j \leq x_n, j \leq p-x_n \right)\\
  \times \sup_{h \in \R} \rmP^p_\calL\left( T_{n-2p} \in [h,h+x_n] \right).
\end{multline*}

Applying \eqref{eq:extendedRWoverRWclaimReversed} we have
\[
  \lim_{n \to  \infty} \sup_{k \leq x_n} \frac{\log \rmP^k_\calL(T_j \geq -x_n, j \leq p-x_n)}{\log n} = -\gamma\left(\tfrac{\sigma_A}{\sigma_Q}\right),\quad \envAS
\]
Moreover, under the measure $\P \otimes \rmP_\calL$, the process $\hat{T}^{n-k}$ has the same law as $T$, thus using Lemma \ref{lem:uniformRwOverRw} for the second term, we have
\begin{equation}
  \label{eq:backward}  \lim_{n \to  \infty} \sup_{k' \leq x_n} \frac{\log  \rmP^p_\calL\left( \hat{T}^{n-p-k'}_j \leq x_n, j \leq p-x_n \right)}{\log n} = -\gamma\left(\tfrac{\sigma_A}{\sigma_Q}\right) \text{ in } \P\text{-probability}.
\end{equation}
Finally, by \eqref{eqn:lltBoundFromAbove}, for any $\epsilon>0$ we have
\[  \lim_{n \to  \infty} \P \left( \sup_{h \in \R}  \frac{\log \rmP^p_\calL \left(T_{n-2p} \in [h,h+x_n] \right)}{\log n} \leq - \frac{1}{2} +\epsilon\right) =1. \]
Combining these three estimates concludes the proof.
\end{proof}

We then prove the lower bound of Theorem \ref{thm:randomExcursion}.
\begin{proof}[Proof of \eqref{eqn:randomLowerExcursion}]
We implement a decomposition similar to \eqref{eqn:decompositionUpperFinal}. However in this case we aim to obtain a bound from below. We set $p = \floor{n/3}$ and fix $0<a<b$. Applying the Markov property at time $p$ we get 
\begin{multline}
  \label{eqn:decomp2}
  \rmP_\calL\left( T_n \geq x-y, T_j \leq x-r_{n,j}, j \leq n \right)\\
  \geq \rmP_\calL\left( T_p \in \left[ -b n^{1/2}, -a n^{1/2}\right], T_j \leq a_n-j^\alpha, j \leq p \right)\qquad \qquad \qquad \qquad\qquad \\
  \times \inf_{z\in [a n^{1/2}, b n^{1/2}]} \rmP_\calL^p \left( T_{n-p} - z \geq x-y, T_j - z \leq x+r_{n,j}, j \leq n-p \right).
\end{multline}
As in the proof of \eqref{eqn:randomUpperExcursion} we use notation $\hat{T}^{n-p}$, defined by $\hat{T}^{n-p}_j = T_{n-p}-T_{n-p-j}$ for the time-reversed random walk. Observe that the condition $T_{n-p}-z \in [x-y,x]$ is weaker than $T_{n-p}-z \in [x-a_n,x]$, and
\begin{multline*}
  \rmP_\calL^p\left( T_{n-p} - z \in[x-a_n, x], T_j-z \leq x+r_{n,j}, j \leq n-p  \right) \\
  \geq \rmP_\calL^p\left( \hat{T}_{n-p}^{n-p} \in[x-a_n + z, x+z], \hat{T}_j^{n-p}\geq -a_n + j^\alpha, j \leq n-p  \right).
\end{multline*}
Applying the Markov property to $\hat{T}^{n-p}$ at time $p$, we have
\begin{multline*}
   \rmP_\calL^p\left( T_{n-p} - z \in[x-a_n, x], T_j-z \leq x-r_{n,j}, j \leq n-p  \right) \\
   \geq \rmP_\calL^p \left( \hat{T}_{p}^{n-p}\in[a n^{1/2}, b n^{1/2}], \hat{T}_j^{n-p} \geq j^\alpha-a_n, j\leq p \right)\qquad \qquad \\
   \times \inf_{z' \in [a n^{1/2},b n^{1/2}]} \rmP_\calL^p\left(\begin{array}{l}\hat{T}^{n-p}_{n-2p} - \hat{T}^{n-p}_p + z'\in[x+z-a_n, x+z]\\ \hat{T}^{n-p}_{j}-\hat{T}^{n-p}_{p} + z'\geq r_{n,p+j}-a_n, j\leq n-2p\end{array} \right).
\end{multline*}
Thus \eqref{eqn:decomp2} yields
\begin{align*}
  &\rmP_\calL \left( T_n \geq x-y, T_j \leq x-r_{n,j}, j \leq n \right)\\
  &\qquad\qquad\geq \rmP_\calL\left( T_p \in \left[ -b n^{1/2}, -a n^{1/2}\right], T_j \leq a_n-j^\alpha, j \leq p \right)\\
  &\qquad\qquad\qquad \times \rmP^p_\calL \left( \hat{T}^n_p \in \left[ a n^{1/2}, b n^{1/2} \right], \hat{T}^n_j \geq -a_n + j^\alpha, j \leq p \right)\times p_n, 
\end{align*}
where 
\[	p_n = \inf_{|z| \leq (b-a)n^{1/2}} \rmP_\calL^p\left(T_{n-2p} \in[x-a_n-z', x-z], T_j \geq -an^{1/2}/2, j\leq n-2p \right).
\]

Using \eqref{eq:extendedRWoverRWclaimReversed} we have
\[
  \lim_{n \to  \infty} \frac{\log \rmP_\calL\left( T_p \in \left[ -b n^{1/2}, -a n^{1/2}\right], T_j \leq a_n - j^\alpha, j \leq p \right)}{\log n}  = -\gamma\left(\tfrac{\sigma_A}{\sigma_Q}\right), \quad \envAS
\]
Further using Theorem \ref{thm:BRWBasicExtended} we handle the second term
\begin{multline*}
	\lim_{n \to  \infty} \frac{\log \rmP^p_\calL \left( \hat{T}_p^n \in \left[ a n^{1/2}, b n^{1/2} \right], \hat{T}_j^n \geq -a_n + j^\alpha, j \leq p \right)}{\log n}  = -\gamma\left(\tfrac{\sigma_A}{\sigma_Q}\right)\\ \quad \text{in } \P\text{-probability.}
\end{multline*}
Note that due to time-reversal, Theorem \ref{thm:BRWBasicExtended} only provides convergence in probability, and Lemma \ref{lem:backward} hints that this convergence cannot be strengthened into almost sure convergence.

Finally, using Lemma \ref{lem:lltBelowWithWall} for any $\epsilon,\eta>0$, choosing $a>0$ large enough, we have
\[
  \lim_{n \to  \infty} \P\left(\frac{\log p_n}{\log n}\geq - \frac{1}{2} -\epsilon \right) \geq 1-\eta.
\]
Combining the last three estimates and letting $\eta \to 0$ concludes the proof.
\end{proof}

\section{Maximal displacement for the branching random walk in random environment}
\label{sec:corrections}

We analyze the asymptotic behavior of $\rmP_\calL(M_n \geq y)$ as $n,y \to  \infty$. Roughly speaking this is done by controlling the number of individuals that are at time $n$ greater than $K_n-\phi \log n$ and stayed below the environment-dependent border $k \mapsto K_k + y$. It is fairly easy to obtain a bound from above by calculating the mean number of particles crossing at some time this boundary. As can be seen in Figure~\ref{fig:img}, with high probability no particle will cross the line. The bound from below requires bounding the second moment and a more subtle analysis.

\begin{figure}[ht]
\centering
\input{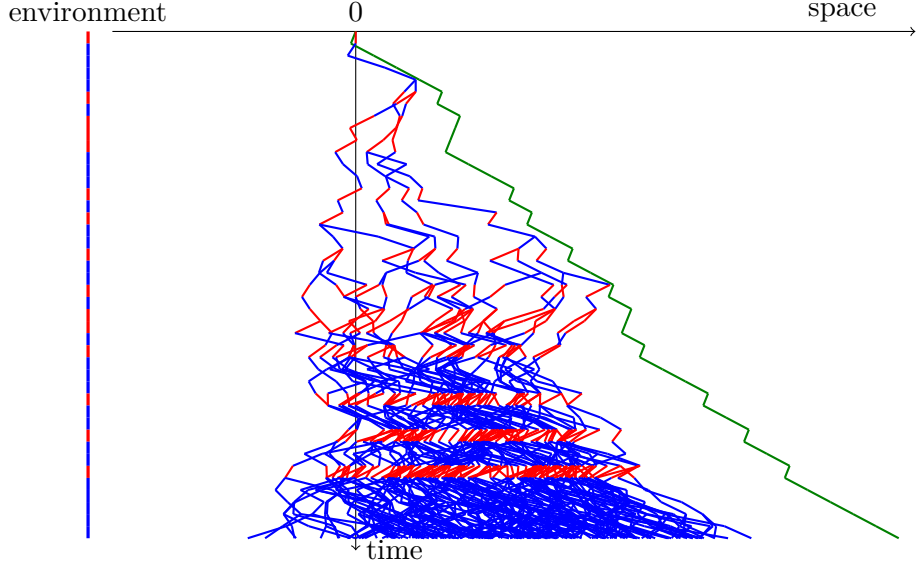}
\caption{A branching random walk in random environment and its border}
\label{fig:img}
\end{figure}

\subsection{Proof of the upper bound of Theorem \ref{thm:logCorrection}}

\begin{lemma}
\label{lem:upperbound}
We assume that \eqref{eqn:defTheta}, \eqref{eqn:defVariances} and \eqref{eqn:integrabilityQuenched} hold. For any $\beta < \phi$, we have
\[
  \rmP_\calL\left( M_n \geq \tfrac{1}{\theta^*}K_n - \beta \log n \right) \to 0 \quad \text{in }\P-\text{probability.}
\]
\end{lemma}

\begin{proof}
For any $n \in \N$ and $\beta > 0$, we set
\begin{equation}\label{eq:ybeta}
	Y_n(\beta) = \sum_{|u|=n} \ind{\theta^* V(u) - K_n \geq - \beta \theta^* \log n} \ind{\theta^* V(u_j) - K_j \leq \log n, j \leq n}.
\end{equation}
We observe that
\[
  \rmP_\calL\left( M_n \geq \frac{K_n}{\theta^*} - \beta \log n \right) \leq \rmP_\calL\left(  \exists u \in \T : V(u) \geq \frac{K_{|u|}+ \log n}{\theta^*} \right) + \rmP_\calL(Y_n(\beta) \geq 1).
\]	
Applying Lemma \ref{lem:frontier} and the Markov inequality, we have
\begin{equation}
  \label{eqn:majoration}
  \rmP_\calL\left( M_n \geq \frac{K_n}{\theta^*} - \beta \log n \right) \leq n^{-1} + \rmE_\calL\left( Y_n(\beta) \right)
\end{equation}
To bound from above $\rmE_\calL\left( Y_n(\beta) \right)$ we use \eqref{eqn:manymodified}, obtaining
\begin{align*}
  \rmE_\calL\left( Y_n(\beta) \right)
  &= \rmE_\calL\left( e^{-\theta^* S_n + K_n} \ind{\theta^* S_n - K_n \geq -\beta \theta^* \log n} \ind{\theta^* S_j - K_j \leq \log n, j \leq n} \right)\\
  &\leq n^{\beta \theta^*} \rmP_\calL\left( T_n \geq -\beta \theta^* \log n, T_j \leq \log n, j \leq n \right).
\end{align*}
By \eqref{eqn:randomUpperExcursion}, for any $\epsilon>0$ we have
\[
  \lim_{n \to  \infty} \P\left( \frac{\rmP_\calL\left( T_n \geq - \beta \theta^* \log n, T_j \leq \log n, j \leq n \right)}{\log n} \leq -\lambda +\epsilon \right) = 1.
\]
As $\phi = \frac{\lambda}{\theta^*}$, we conclude
\[
  \lim_{n \to  \infty} \P\left( \left\{ \rmP_\calL\left( T_n \geq - \beta \theta^* \log n, T_j \leq \log n, j \leq n \right) \leq n^{-\theta^*\phi +\epsilon} \right\}\right) = 1.
\]
Setting $\epsilon = \frac{\theta^*}{2} (\phi - \beta)>0$, we have $\lim_{n \to  \infty} \P\left( \left\{\rmE_\calL\left( Y_n(\beta) \right) \leq n^{-\epsilon} \right\}\right) = 1$.
Therefore, $\rmE_\calL\left( Y_n(\beta) \right)$ converges to $0$ in $\P$-probability. By \eqref{eqn:majoration} we conclude the proof.
\end{proof}

\subsection{Proof of the lower bound of Theorem \ref{thm:logCorrection}}

To prove the lower bound of Theorem \ref{thm:logCorrection}, we first prove that the probability to observe individuals above $\frac{K_n}{\theta^*} - \phi \log n$ does not decrease too fast. Secondly, we use the fact that the population grows at exponential rate to conclude the proof.
\begin{lemma}
\label{lem:lowerbound}
We assume that \eqref{eqn:defTheta}, \eqref{eqn:secondMoment}, \eqref{eqn:defVariances} and \eqref{eqn:integrabilityQuenched} hold. For any $\epsilon > 0$, we have
\[
  \lim_{n \to  \infty} \P\left( \left\{ \rmP_\calL \left( M_n \geq \frac{K_n}{\theta^*} - \phi \log n \right) \geq n^{-\epsilon} \right\} \right) = 1.
\]
\end{lemma}

\begin{proof}
In the proof we use the second moment method. We introduce an environment-dependent path which individuals are disallowed to cross. This enforces the first and second moment to have the same behavior. Let $\delta > 0$, for $k \leq n$ we set 
\[r_{n,k} =
\begin{cases}
  k^{1/3} - \delta \log n & \mathrm{if} \quad k \in \{1, \ldots, \floor{n/2}\} \\
  (n-k)^{1/3} +(\phi -\delta) \log n & \mathrm{if} \quad k \in \{\floor{n/2}+1, \ldots, n\}.
\end{cases}
\]
In $X_n(\delta)$ we count particles near $K_n/\theta^* - \varphi \log n$ who stayed below\footnote{Note that the choice of the exponent $1/3$ in the definition of $r_{n,k}$ is arbitrary, any exponent $\gamma \in (0,1/2)$ would yield a similar result, using the result in Theorem \ref{thm:randomExcursion}.} $K_j/\theta^* -r_{n,j}$ at all time $j \leq n$ viz.
\[
  X_n (\delta) = \sum_{|u|=n} \ind{V(u) - K_n/\theta^* + \phi \log n \in [0,\delta \log n], V(u_j) \leq K_j/\theta^* - r_{n,j}}.
\]
We use the Cauchy-Schwarz inequality to get 
\begin{equation}
  \label{eqn:cauchySchwarz}
  \rmP_\calL\left( M_n \geq \frac{K_n}{\theta^*} - \beta \log n \right)\geq \rmP_\calL(X_n(\delta) \geq 1) \geq \frac{\left(\rmE_\calL X_n(\delta)\right)^2}{\rmE_\calL\left( X_n(\delta)^2 \right)}.  
\end{equation}

\begin{figure}
\centering
\begin{tikzpicture}
\draw [->] (0,-1) -- (0,6) node[left] {space};

\draw [->] (-0.2,0) node [left] {0} -- (10,0) node[below] {time};
\draw [color = red] (0,6) node[right] {$\frac{K_n}{\theta^*} + y$};

\draw [very thick, color=red] ( 0,0.5 ) -- (0.5,0.533) -- (1.0,1.278) -- (1.5,1.957) -- (2.0,1.977) -- (2.5,2.383) -- (3.0,2.728) -- (3.5,2.97) -- (4.0,3.102) -- (4.5,3.265) -- (5.0,3.738) -- (5.5,3.896) -- (6.0,3.992) -- (6.5,4.638) -- (7.0,5.06) -- (7.5,5.336) -- (8.0,5.4) -- (8.5,5.532) -- (9.0,5.439);

\draw [<->,thick, color=green]  (9.0,5.44) -- (9.0,4.56);
\draw [color = green] (9,5) node[right] {$\phi \log n$};
\draw [<->,thick, color=red]  (9.0,4.2) -- (9.0,4.56);
\draw [color = red] (9,4.38) node[right] {$\delta \log n$};

\draw [color = blue, densely dashed, domain=0:9,samples=100] plot(\x, {0.3-0.3*sqrt((\x+0.1)*(9.1-\x))+0.55*\x + 0.5});
\draw [<->,color = blue] (4.5,3.265) -- (4.5,1.85);
\draw [color = blue] (4.5,2.55) node[right] {$r_{n,j}$};
\draw ( 0,0 ) -- (0.5,0.184) -- (1.0,0.209) -- (1.5,0.258) -- (2.0,0.302) -- (2.5,0.341) -- (3.0,0.362) -- (3.5,0.56) -- (4.0,1.133) -- (4.5,1.253) -- (5.0,1.978) -- (5.5,2.152) -- (6.0,2.212) -- (6.5,2.974) -- (7.0,3.104) -- (7.5,3.58) -- (8.0,4.138) -- (8.5,4.181) -- (9.0,4.25);
\draw (8.5,4.181) -- (9.0,4.55);
\draw ( 0,0 ) -- (0.5,0.19) -- (1.0,0.213) -- (1.5,0.338) -- (2.0,0.523) -- (2.5,0.613) -- (3.0,0.635) -- (3.5,0.841) -- (4.0,1.057) -- (4.5,1.452) -- (5.0,1.533) -- (5.5,1.899) -- (6.0,2.134) -- (6.5,2.321) -- (7.0,2.731) -- (7.5,3.438) -- (8.0,4.024) -- (8.5,4.406) -- (9.0,4.52);
\draw ( 0,0 ) -- (0.5,0.088) -- (1.0,0.125) -- (1.5,0.178) -- (2.0,0.287) -- (2.5,0.42) -- (3.0,0.423) -- (3.5,0.98) -- (4.0,1.339) -- (4.5,1.551) -- (5.0,1.824) -- (5.5,1.975) -- (6.0,2.151) -- (6.5,2.194) -- (7.0,2.196) -- (7.5,2.8) -- (8.0,3.548) -- (8.5,4.276) -- (9.0,4.37);
\draw [densely dotted] (0.5,0.088) -- (1, - 0.2);
\draw [densely dotted] (6.0,2.151) -- (6.5,1.89);
\draw [densely dotted] (8.0,3.548) -- (8.5,3.5);
\draw [densely dotted] (3.5,0.56) -- (4,0.7);
\draw [densely dotted] (6.5,2.321) -- (7,2.1);
\draw [densely dotted] (7.0,2.196) -- (7.5,2.4);
\end{tikzpicture}
\caption{Some trajectories counted in $X_n(\delta)$}
\end{figure}
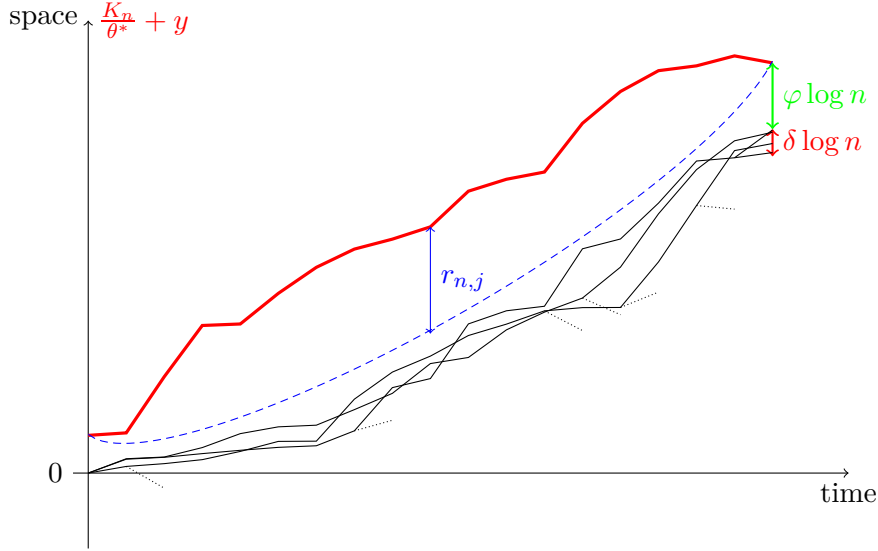

We first bound from below $\rmE_\calL(X_n(\delta))$. Applying \eqref{eqn:manytoone} and recalling \eqref{eqn:defT}, we have
\begin{align*}
  \rmE_\calL\left(X_n(\delta)\right) &=  \rmE_\calL\left( e^{-T_n} \ind{T_n/\theta^* + \phi \log n \in \left[0,\delta \log n \right]} \ind{T_j/\theta^* \leq -  r_{n,j}, j \leq n} \right)\\
  &\geq n^{\theta^* \phi - \theta^* \delta} \rmP_\calL\left( T_n/\theta^* +  \phi \log n \in \left[0,\delta  \log n\right], T_j/\theta^* \leq -  r_{n,j} , j \leq n \right).
\end{align*}
Using \eqref{eqn:randomUpperExcursion}, for any $\epsilon>0$ we have
\begin{equation}
  \label{eqn:mean}
  \lim_{n \to  \infty} \P\left( \rmE_\calL \left( X_n(\delta) \right) \geq n^{-\theta^* \delta-\epsilon} \right) = 1.
\end{equation}

Further, we bound from above $\rmE\left( X_n(\delta)^2 \right)$. We note that $X_n(\delta)^2$ is the number of pairs of individuals that are at time $n$ in a neighborhood of $K_n/\theta^* - \phi \log n$ and stayed at any time $k \leq n$ at distance at least $r_{n,k}$ from $K_k/\theta^*$. We partition this set of pairs $(u^1,u^2) \in \T$ according to the most recent common ancestor, denoted by $u^1 \wedge u^2$. More precisely, $X_n(\delta)^2 = \sum_{k=0}^{n} \Lambda_{k}$,  where 
\begin{equation*}
	\Lambda_k = \sum_{|u| = k} \sum_{\substack{|u^1| = |u^2| = n\\ u^1 \wedge u^2 = u}} \ind{V(u^i) - K_n/\theta^* + \phi \log n \in [0,\delta \log n], V(u^i_j) \leq K_j/\theta^* - r_{n,j}, i \in \{1,2\},j\leq n}.
\end{equation*}
We notice that $\Lambda_n = X_n (\delta)$. For $k<n$ we study $\Lambda_k$ applying the Markov property at time $k+1$, namely we denote $\calF_k = \sigma(u,V(u), |u| \leq k)$ and calculate
\begin{equation}
	\rmE_\calL\left( \left.\Lambda_k \right| \calF_{k+1} \right) \leq \sum_{|u|=k} \ind{V(u_j) \leq K_j/\theta^* -r_{n,j}, j \leq k} \sum_{\substack{|u^1| = |u^2| = k+1\\ u^1 \wedge u^2 = u}} f_{k+1}(V(u^1)) f_{k+1}(V(u^2)),\label{eq:LambdaDecomposed}
\end{equation}
where
\[
  f_{k+1}(x) =
  \rmE^{k+1}_\calL\left( \sum_{|u|=n-k-1} \ind{\begin{array}{l}\scriptstyle V(u) - K_n/\theta^*+x + \phi \log n \in [0,\delta \log n]\\ \scriptstyle V(u_j)-{K_{k+j+1}}/{\theta^*}+x \leq  - r_{n,k+j+1}, j \leq n-k-1\end{array}} \right).
\]
Note that if $x \geq K_{k+1}/\theta^* - r_{n,k+1}$ then $f_{k+1}(x) = 0$.

We recall that under law $\rmP^k_\calL$, for any $n \in \N$ we write $T_n = \theta^* S_n - \sum_{j=k+1}^{k+n} \kappa_j(\theta^*)$ and $\sum_{j=k+1}^{k+n} \kappa_j(\theta^*)= K_{k+n}-K_k$. Applying \eqref{eqn:manytooneGeneral} we have
\[
  f_{k+1}(x) \leq \rmE^{k+1}_\calL\left( e^{- T_{n-k-1}} \ind{\begin{array}{l} \scriptstyle T_{n-k-1} + \theta^*x-K_{k+1} + \theta^*\phi \log n \in [0, \theta^*\delta \log n]\\ \scriptstyle T_j \leq  K_{k+1}-\theta^*x - \theta^*r_{n,j+k+1}, j \leq n-k-1\end{array}} \right).
\]
Observe that for any $j \leq n$, $r_{n,j}\geq -\delta \log n$. We have
\begin{equation}
	f_{k+1}(x) \leq n^{\theta^* \phi} e^{\theta^* x - K_{k+1}} \rmP^k_\calL\left( \begin{array}{l}  T_{n-k-1}- K_{k+1} + \theta^* x  + \theta^*\phi \log n \in [0,\theta^*\delta \log n]\\  T_j \leq K_{k+1}-\theta^*x + 2 \theta^*\delta \log n, j \leq n-k-1 \end{array} \right). \label{eq:fk1estimate}
\end{equation}
We set $b_n = \ceil{(\log n)^6}$, and bound from above $\rmE_\calL\left(\Lambda_k\right)$ in three different manners depending whether $k \leq b_n$, $k \geq n-b_n$ or $k \in [b_n,n-b_n]$. We write
\begin{multline*}
  \Phi^\text{start}_n = \max_{k \leq b_n} \rmE_\calL\left( \left(\sum_{\ell \in L_{k}}  e^{\theta^* \ell - \kappa_{k}(\theta)} \right)^2 \right), \ \Phi^\text{end}_n = \max_{k \in [n-b_n,n]} \rmE_\calL\left( \left(\sum_{\ell \in L_{k}}  e^{\theta^* \ell - \kappa_{k}(\theta)} \right)^2 \right)\\
  \text{ and } \Phi_n = \max_{k \leq n} \rmE_\calL\left( \left(\sum_{\ell \in L_{k}}  e^{\theta^* \ell - \kappa_{k}(\theta)} \right)^2 \right).
\end{multline*}
By \eqref{eq:fk1estimate}, $f_{k+1}(V(u) + \ell) \leq n^{\theta^* \phi} e^{\theta^* \ell - \kappa_{k+1}(\theta^*)} e^{\theta^* V(u) - K_k}$ for $k \in [b_n,n-b_n]$ thus \eqref{eq:LambdaDecomposed} yields
\begin{align}
  &\rmE_\calL\left( \Lambda_k \right)\nonumber\\
  \leq &n^{2\theta^* \phi} \rmE_\calL\left( \sum_{\ell, \ell' \in L_{k+1}} e^{\theta^* \left( \ell + \ell'\right) - 2 \kappa_{k+1}(\theta^*)} \right)\rmE_\calL\left( \sum_{|u|=k} e^{2\left(\theta^* V(u_k) - K_k\right)} \ind{V(u_j) \leq K_j/\theta^* -r_{n,j}, j \leq k} \right) \nonumber\\
  \leq & n^{2\theta^* \phi} \Phi_n\rmE_\calL\left( e^{T_k} \ind{T_j \leq - \theta^* r_{n,j}, j \leq k} \right) \label{eq:joinedEstimates}.
\end{align}
by \eqref{eqn:manytoone}. For $k\in[b_n,n-b_n]$ we have $r_{n,k} \geq (\log n)^2-\delta \log n$ thus we obtain
\begin{equation}
  \label{eqn:gestionMiddlePart}
  \rmE_\calL\left( \Lambda_k \right) \leq \Phi_n n^{\theta^* (2\phi-\delta)} e^{-\theta^* (\log n)^2}.
\end{equation}

Secondly, we consider $k \in [n-b_n,n-1]$. We have $r_{n,k} \geq (\phi-\delta) \log n$ and analogously to \eqref{eq:joinedEstimates} we get
\[
  \rmE_\calL\left( \Lambda_k \right) \leq  \Phi^\text{end}_n n^{2\theta^* \phi} \rmE_\calL\left( e^{T_k} \ind{T_j \leq - \theta^* r_{n,j}, j \leq k} \right),
\]
we decompose this expectation depending on the endpoint $T_k$, we obtain
\begin{align}
  \rmE_\calL\left( \Lambda_k \right) &\leq  \Phi^\text{end}_n n^{2\theta^* \phi} \rmE_\calL\left( e^{T_k} \ind{T_j \leq - \theta^* r_{n,j}, j \leq k} (\ind{T_k \leq -b_n} + \ind{T_k \geq -b_n}) \right)\nonumber\\
  &\leq \Phi^\text{end}_n n^{2\theta^* \phi} \left( e^{-b_n} + n^{-\theta^*(\phi - \delta)} \rmP_\calL\left( T_k \geq -b_n, T_j \leq - \theta^* r_{n,j}, j \leq k\right) \right) \nonumber\\
  &\leq \Phi^\text{end}_n \left( n^{\theta^* (\phi + \delta)} P^\text{end}_n +  e^{-\theta^*(\log n)^2}\right), \label{eqn:gestionEndPart}
\end{align}
for $n \geq 1$ large enough, where
\[
  P^\text{end}_n = \max_{k \in \{n - b_n +1, \ldots,n \} } \rmP_\calL\left( T_k \geq -b_n,  T_j \leq -\theta^* r_{n,j}, j \leq k \right).
\]
Finally, we deal with $k\leq b_n$. We use \eqref{eq:fk1estimate} to obtain 
\[
  f_{k+1}(x) \leq n^{\theta^* \phi} e^{\theta^* x - K_{k+1}} \left(P^\text{start}_n \ind{\theta^* x - K_{k+1} \geq -b_n} + \ind{\theta^* x - K_{k+1} \leq -b_n} \right),
\]
where
\[
  P^\text{start}_n = \sup_{k \leq b_n} \sup_{y \leq b_n} \rmP^k_\calL\left(
  \begin{array}{l}
    T_{n-k-1} + \theta^*(\phi \log n - y) \in [0,\delta \log n],\\
    T_j \leq \theta^*(2 \delta \log n + y), j \leq n-k-1
  \end{array}
  \right).
\]
Applying \eqref{eqn:manytoone} similarly to \eqref{eq:joinedEstimates} and recalling that $r_{n,k} \geq -\delta \log n$ we get
\[
    \rmE_\calL\left( \Lambda_k \right) \leq \Phi^\text{start}_n n^{2\theta^* \phi} \rmE_\calL\left( e^{T_k} \left((P^\text{start}_n)^2\ind{T_k \leq \theta^*\delta \log n} + \ind{T_k \leq - b_n}  \right)\right).
\]
Thus for all $n$ large enough,
\begin{equation}
  \label{eqn:gestionStartPart}
  \rmE_\calL\left( \Lambda_k\right) \leq \Phi^\text{start}_n\left[ n^{2\theta^* (\phi + \delta)}(P^\text{start}_n)^2 + e^{-\theta^*(\log n)^2/2}\right].
\end{equation}
As $\rmE_\calL\left( X_n(\delta)^2 \right) = \sum_{k=1}^n \rmE_\calL\left( \Lambda_k\right)$, gathering \eqref{eqn:gestionMiddlePart}, \eqref{eqn:gestionEndPart} and \eqref{eqn:gestionStartPart}, we obtain
\begin{multline*}
  \rmE_\calL\left( X_n(\delta)^2 \right) \leq b_n \Phi^\text{start}_n \left(n^{2 \theta^* (\phi + \delta)} (P^\text{start}_n)^2 + e^{-\theta^*(\log n)^2/2}\right) + n^{1+\theta^* (2\phi-\delta)}  \Phi_n e^{-\theta^*(\log n)^2/2}\\
  + b_n \Phi^\text{end}_n \left( n^{\theta^* (\phi + \delta)} P^\text{end}_n  + e^{-\theta^*(\log n)^2/2} \right) + \rmE_\calL \left( X_n(\delta) \right).
\end{multline*}
Using \eqref{eqn:secondMoment}, we observe that for any $\eta>0$ there exists $x \geq 0$ such that for any $n \in \N$
\[
  \P\left( \Phi^\text{start}_n \geq x b_n \right) + \P\left( \Phi^\text{end}_n \geq x b_n \right) + \P\left( \Phi_n \geq x n \right) \leq \eta.
\]
Further, by \eqref{eqn:randomUpperExcursion} and \eqref{eqn:defLambda}, for any $\epsilon>0$,
\[
  \lim_{n \to  \infty} \P\left( P^\text{start}_n \leq n^{-\theta^* \phi + \epsilon} \right) = 1, \quad  \lim_{n \to  \infty} \P\left( P^\text{end}_n \leq n^{-\theta^* \phi + \epsilon} \right) = 1.
\]
We recall \eqref{eq:ybeta}, we have $\E(X_n(\delta)) \leq \E(Y_n(\delta))$. Applying Lemma \ref{lem:upperbound}, we conclude that for any $\epsilon>0$ and $\eta > 0$, $\liminf_{n \to  \infty} \P\left( \rmE_\calL\left( X_n(\delta)^2 \right) \leq n^{2\theta^*\delta + 2\epsilon} \right) \geq 1-\eta.$
Letting $\eta \to 0$ we obtain
\begin{equation}
  \label{eqn:variance}
  \lim_{n \to  \infty} \P\left(\rmE_\calL\left( X_n(\delta)^2 \right) \leq n^{2\theta^*\delta + 2\epsilon} \right) = 1.
\end{equation}
Finally using \eqref{eqn:cauchySchwarz}, \eqref{eqn:mean} and \eqref{eqn:variance} for any $\delta > 0$ we have
\[
  \lim_{n \to  \infty} \P\left( \rmP_\calL\left( M_n \geq \frac{K_n}{\theta^*} - \phi \log n \right) \geq n^{-4\theta^* \delta -4\epsilon} \right) = 1.
\]
Choosing $\epsilon,\delta >0$ small enough we conclude the proof.
\end{proof}

\begin{proof}[Proof of Theorem \ref{thm:logCorrection}]
Lemma \ref{lem:upperbound} covers the case of $\beta < \phi$. We are now left to prove that for $\beta > \phi$ we have
\[
  \lim_{n \to  \infty} \rmP_\calL\left( M_n \geq \frac{K_n}{\theta^*} - \beta \log n \right) = 1 \quad \text{in } \P\text{-probability}.
\]
To do so, we use the fact that the population grows at exponential rate, and that each individual in the branching random walk alive at a given time starts an independent branching random walk.

Let $A \in \N$, we set $\T^A$ the subtree of $\T$ consisting in the set of particles that never made a jump smaller than $-A$. We also trim the tree so that the maximal number of offspring is $A$ (for example by choosing the $A$ largest children). We denote by $N_n^A = \#\left\{ u \in \T^A : |u| =n \right\}$. For $A$ large enough, by \cite[Theorem 5.5 (iii)]{Tanny:1977aa} and \eqref{eqn:supercritical}, there exists a constant $\rho_A>1$  such that
\begin{equation}
\liminf_{n\to \infty}\left(N_{n}^{A}\right)^{1/n}=\rho_{A},\quad \envAS\label{eq:fastgrowth}
\end{equation}
on the set $\mathcal{A}_{A}=\left\{ \lim_{n\to \infty}N_n^A = \infty\right\}$. Observe that $V(u) \geq - A |u|$ for any $u \in \T^A$.

For $n \in \N$, we set $\epsilon = \frac{\beta-\phi}{4}$ and $k = \ceil{\epsilon \log n}$ and $\mathcal{B}_{n}=\left\{ |K_{k}|\leq\epsilon\theta^{*}\log n\right\} $. As each of the individuals alive at generation $k$ starts an independent BRWre with environment $(\calL_{k+j}, j \in \N)$, applying the Markov property at time $k$,
\begin{align*}
1_{\mathcal{B}_{n}}\rmP_\calL\left(M_{n}\leq K_{n}/\theta^{*}-\beta\log n\right) & \leq1_{\mathcal{B}_{n}}\rmP_\calL^{k}\left(M_{n-k}\leq K_{n}/\theta^{*}-\beta\log n+Ak\right)^{N_{k}^{A}}\\
 & \leq1_{\mathcal{B}_{n}}\rmP_\calL^{k}\left(M_{n-k}\leq(K_{n}-K_{k})/\theta^{*}-\varphi\log(n-k)\right)^{N_{k}^{A}},
\end{align*}
\envAS for $n$ large enough. Let $r_A \in (1,\rho_A)$ and denote the event 
\[
\mathcal{C}_{n}=\left\{ \rmP_\calL^{k}\left(M_{n-k}\leq(K_{n}-K_{k})/\theta^{*}-\varphi\log(n-k)\right)\leq1-r_{A}^{-k}\right\} .
\]
By \eqref{eq:fastgrowth}, we conclude that
\[
1_{\mathcal{A}_{A}\cap\mathcal{B}_{n}\cap\mathcal{C}_{n}}\rmP_\calL\left(M_{n}\leq K_{n}/\theta^{*}-\beta\log n\right)\to0,\quad \envAS
\]
By translation invariance and Lemma \ref{lem:lowerbound} we have $\lim_{n \to  \infty} \P(\mathcal{C}_{n}) = 1$. As $|K_k|< \infty$ \envAS we have $\lim_{n \to  \infty} \P(\mathcal{B}_{n})=1$. This yields
\[
1_{\mathcal{A}_{A}}\rmP_\calL\left(M_{n}\leq K_{n}/\theta^{*}-\beta\log n\right)\to0,\quad\text{in }\P\text{-probability}.
\]
Finally, we observe that when $A\nearrow \infty$ we have $1_{\mathcal{A}_{A}}\nearrow1$
by the assumption in \eqref{eqn:nonExtinction}.
\end{proof}

\appendix

\section{Precise behavior of the median and tightness} \label{sec:tightness}
In this section we will discuss how the results of Theorem \ref{thm:logCorrection} can be restated and refined. We recall that $m^Q_n$, defined in \eqref{eqn:defQuenchedMedian}, is the median of $M_n$ conditionally on the environment. We would like to describe the asymptotic behavior of $m^Q_n$ as simply as possible. We recall that $(K_n,n\in \N)$ is a random walk measurable with respect to the environment. Theorem~\ref{thm:logCorrection} yields that $m^Q_n \approx K_n$ and more precisely
\[
	m^Q_n = K_n - \phi \log n + o_P(\log n). 
\]
However, the sequence hidden in $o_P$ is not trivial i.e.  $m^Q_n - K_n +\phi \log n$ does not converge \envAS (see Lemma \ref{lem:backward}). This happens because every now and again the environment can speed up (or slow down) every individual in the process in a way not captured correctly by $K_n$. It is not clear for us if there exists a simple function of the environment which describes $m^Q_n$ more precisely. Finding such a function might be an interesting research task. In this section we discuss another question: the tightness of $(M_n-m^Q_n, n \in \N)$.

\subsection{Quenched tightness} We fix the environment $\calL$ and ask whether the sequence $\{\Delta_n\}_{n\geq 1}$ defined by
\[
	\Delta_n = M_n - m^Q_n,
\]
is tight (with respect to $\rmP_{\calL}$). For some cases this question has been answered positively in \cite{Fan12}. We skip lengthy description of conditions referring the reader to \cite[Section 2]{Fan12} and \cite[Section 5]{Fan12}. Instead, we brief in words the ones of \cite[Section 2]{Fan12}. We require that the branching law has uniformly bounded support and the mean offspring number is uniformly bounded away from $1$. The law of displacements need to be such that there exists uniform $x_0$ such that particles moves above $x_0$ with high probability. Further, the marginal of the displacement decays exponentially (with an uniform exponent). Finally, we assume that with high and uniform probability all particles born in one event stay in a ball of an uniform size. The setting of \cite{Fan12} is quite general and thus these conditions are somewhat restrictive. We suspect uniformization effects in our case of i.i.d. environments. We are quite convinced that the condition about branching laws can be relaxed and that it is enough to assume only mild moment conditions. The situation of the displacement law is less clear. Let us illustrate this on a concrete example. Consider a system with the dyadic branching and the displacements being $\mathcal{N}(0, \sigma^2_n)$,  where $\{\sigma_n^2\}$ is an i.i.d. environment sampled from a exponentially integrable distribution with an unbounded support. We pose as an open question determining whether this $\{\Delta_n\}$ is tight. 

\subsection{Annealed tightness} Another way is to study the tightness in the annealed setting. One way to state this question would be to ask whether the sequence 
\[
	\{M_n - \rmE_{\calL} M_n\}_{n\geq 0},
\]
is tight with respect to the joint law of the environment and the branching law i.e. $\P$. This is a weaker notion than the quenched tightness and we suspect this tightness may hold under quite general conditions. Before discussing further, let us present a vanilla version of the Host-Dekking argument \cite{Dekking:1991aa}.

\begin{fact}
Assume that $\P(\#\{ \ell \in L_1 \} \geq 2 )=1$ and \mbox{$C=\E \min_{l\in L_1} l>-\infty$}. Then for $n\geq 1$ we have
\[
\E|M_{n}-{\rm E}_{\calL}M_{n}|\leq 2 |\E M_{n+1} - \E M_{n} - C|.
\]
\end{fact}
\begin{proof}  Let $L_{1}$ be a realization of $\mathcal{L}_{1}$ and for any $l\in L_{1}$ let $M_{n}^{l}$ denote
the maximum of the sub-system starting from $l$ (relatively to $l$).
Obviously we have
\[
M_{n}=\max_{l\in L_{1}}\{l+M_{n}^{l}\}\geq\max_{l\in L_{1}}\{M_{n}^{l}\}+\min_{l\in L_{1}}l.
\]
Now we use the first assumption. Let $l_{1},l_{2}$ denote two atoms of $L_{1}$. We write
\[
\E M_{n}\geq\E\max \left\{ M_{n}^{l_{1}},M_{n}^{l_{2}} \right\}+\E\min_{l\in L_{1}}l,
\]
and use the equality $\max(a,b)=(a+b+|a-b|)/2$ obtaining
\[
\E M_{n}\geq\E M_{n}^{l_{1}}+\frac{1}{2}\E|M_{n}^{l_{1}}-M_{n}^{l_{2}}|+\E\min_{l\in L_{1}}l.
\]
Rearranging we obtain
\[
 \E|M_{n}^{l_{1}}-M_{n}^{l_{2}}|\leq2\left(\E M_{n}-\E M_{n}^{l_{1}}-\E\min_{l\in L_{1}}l\right).
\]
The crucial point of the argument is using the time homogeneity of the environment with respect to $\E$. This implies $\E M_{n}^{l_{1}}=\E M_{n-1}$. We notice that conditionally on $\calL$ the random variables $M^{l_1}_n, M^{l_2}_n$ are independent and have the same distribution. Using this, the Jensen inequality and the time homogeneity again we obtain
\[
  \E|M_{n}^{l_{1}}-M_{n}^{l_{2}}|\geq\E|M_{n}^{l_{1}}-\rmE_{\calL}M_{n}^{l_{1}}|= \E|M_{n-1}-\rmE_{\calL}M_{n-1}|.
\]
This concludes the proof.
\end{proof}
The assumption $\P(\text{number of individuals in } L_1\geq 2 )=1$  is made only to simplify the argument and the condition $C=\E \min_{l\in L_1} l>-\infty$ is non-restrictive. Consequently, the usefulness of the fact depends on the quality of estimates for $\E M_{n+1} - \E M_n$. A very simple case is to assume that there exists $K$ such that $\P(\forall_{l\in L_{1}}|l|\leq K)=1$. Then obviously \mbox{$\E M_{n+1} - \E M_n\leq K$} and the tightness follows. For a general case one can use a standard sub-additivity argument showing $\E M_n/n \to c \in \R$ and, thus, the existence of a constant $C_1$ and a sequence  $\{n_k\}$ having positive density such that $\E M_{n_k+1} - \E M_{n_k} \leq C_1$. Consequently, the tightness holds on  the subsequence $\{n_k\}$. The usual way to obtain a better result is to establish precise estimates for $\E M_n$. Refining our methods we can deduce that $\E M_n = c n -\varphi \log n + o(\log n)$ implying that 
\[
	\E|M_{n}-{\rm E}_{\calL}M_{n}| = o(\log n).
\]

\subsection*{Acknowledgements} We thank prof. Anton Bovier, prof. Zhan Shi and prof. Ofer Zeitouni for useful discussions and various suggestions. 

\bibliographystyle{plain}

\end{document}